\documentclass[11pt,reqno]{amsart}
\usepackage{amsthm}
\usepackage{amssymb}
\usepackage{latexsym}
\usepackage{multicol}
\usepackage{verbatim,enumerate}
\usepackage{amscd}
\usepackage[all]{xy}

\usepackage[usenames]{color}
\usepackage{hyperref}
\usepackage{amsmath, amscd}

\advance\textwidth by 1.2in \advance\oddsidemargin by -.6in \advance\evensidemargin by -.6in
\newtheorem*{cor}{Corollary}
\newtheorem*{lem}{Lemma}
\newtheorem*{prop}{Proposition}

\theoremstyle{definition}

\theoremstyle{definition}
\newtheorem{thm}{Theorem}
\newtheorem*{thm*}{Theorem}
\newtheorem*{conj}{Conjecture}
\newtheorem*{rem}{Remark}

\newcounter{cnt}
\newenvironment{enumerit}{\begin{list}{{\hfill\rm(\roman{cnt})\hfill}}{%
\settowidth{\labelwidth}{{\rm(iv)}}\leftmargin=\labelwidth%
\advance\leftmargin by \labelsep\rightmargin=0pt\usecounter{cnt}}}{\end{list}} \makeatletter
\def\mydggeometry{\makeatletter\dg@YGRID=1\dg@XGRID=20\unitlength=0.003pt\makeatother}
\makeatother \theoremstyle{remark}

\numberwithin{equation}{section}

\let\bwdg\bigwedge
\def\bigwedge{{\textstyle\bwdg}}

\begin{document}
\newcommand{\semi}{{\,\rule[.1pt]{.4pt}{5.3pt}\hskip-1.9pt\times}}
\newcommand{\semc}{{\rm sc}}
\newcommand{\Lhg}{\widehat{\mathfrak{g}}}
\newcommand{\Lhh}{\widehat{\mathfrak{h}}}
\newcommand{\Lhb}{\widehat{\mathfrak{b}}}
\newcommand{\wDelta}{\widehat{{\Delta}}}
\newcommand{\wPhi}{\widehat{{\Phi}}}
\newcommand{\ph}{\widehat{P}}
\newcommand{\charc}{\hbox{\rm Char}\,}
\newcommand{\waff}{{W^{\rm aff}}}
\newcommand{\wiff}{{\widetilde{W}^{\rm aff}}}
\newcommand{\thmref}[1]{Theorem~\ref{#1}}
\newcommand{\secref}[1]{Section~\ref{#1}}
\newcommand{\lemref}[1]{Lemma~\ref{#1}}
\newcommand{\propref}[1]{Proposition~\ref{#1}}
\newcommand{\corref}[1]{Corollary~\ref{#1}}
\newcommand{\remref}[1]{Remark~\ref{#1}}
\newcommand{\defref}[1]{Definition~\ref{#1}}
\newcommand{\er}[1]{(\ref{#1})}
\newcommand{\id}{\operatorname{id}}
\newcommand{\ord}{\operatorname{\emph{ord}}}
\newcommand{\sgn}{\operatorname{sgn}}
\newcommand{\wt}{\operatorname{wt}}
\newcommand{\tensor}{\otimes}
\newcommand{\from}{\leftarrow}
\newcommand{\nc}{\newcommand}
\newcommand{\rnc}{\renewcommand}
\newcommand{\dist}{\operatorname{dist}}
\newcommand{\qbinom}[2]{\genfrac[]{0pt}0{#1}{#2}}
\nc{\cal}{\mathcal} \nc{\goth}{\mathfrak} \rnc{\bold}{\mathbf}
\renewcommand{\frak}{\mathfrak}
\newcommand{\supp}{\operatorname{supp}}
\newcommand{\Irr}{\operatorname{Irr}}
\newcommand{\psym}{\mathcal{P}^+_{K,n}}
\newcommand{\psyml}{\mathcal{P}^+_{K,\lambda}}
\newcommand{\psymt}{\mathcal{P}^+_{2,\lambda}}
\renewcommand{\Bbb}{\mathbb}
\nc\bomega{{\mbox{\boldmath $\omega$}}} \nc\bpsi{{\mbox{\boldmath $\Psi$}}}
 \nc\balpha{{\mbox{\boldmath $\alpha$}}}
 \nc\bpi{{\mbox{\boldmath $\pi$}}}
  \nc\bxi{{\mbox{\boldmath $\xi$}}}
\nc\bmu{{\mbox{\boldmath $\mu$}}} \nc\bcN{{\mbox{\boldmath $\cal{N}$}}} \nc\bcm{{\mbox{\boldmath $\cal{M}$}}} \nc\blambda{{\mbox{\boldmath
$\lambda$}}}\nc\bnu{{\mbox{\boldmath $\nu$}}}

\newcommand{\Tmn}{\bold{T}_{\lambda^1, \lambda^2}^{\nu}}

\newcommand{\lie}[1]{\mathfrak{#1}}
\newcommand{\ol}[1]{\overline{#1}}
\makeatletter
\def\section{\def\@secnumfont{\mdseries}\@startsection{section}{1}%
  \z@{.7\linespacing\@plus\linespacing}{.5\linespacing}%
  {\normalfont\scshape\centering}}
\def\subsection{\def\@secnumfont{\bfseries}\@startsection{subsection}{2}%
  {\parindent}{.5\linespacing\@plus.7\linespacing}{-.5em}%
  {\normalfont\bfseries}}
\makeatother
\def\subl#1{\subsection{}\label{#1}}
 \nc{\Hom}{\operatorname{Hom}}
  \nc{\mode}{\operatorname{mod}}
\nc{\End}{\operatorname{End}} \nc{\wh}[1]{\widehat{#1}} \nc{\Ext}{\operatorname{Ext}} \nc{\ch}{\text{ch}} \nc{\ev}{\operatorname{ev}}
\nc{\Ob}{\operatorname{Ob}} \nc{\soc}{\operatorname{soc}} \nc{\rad}{\operatorname{rad}} \nc{\head}{\operatorname{head}}
\def\Im{\operatorname{Im}}
\def\gr{\operatorname{gr}}
\def\mult{\operatorname{mult}}
\def\Max{\operatorname{Max}}
\def\ann{\operatorname{Ann}}
\def\sym{\operatorname{sym}}
\def\loc{\operatorname{loc}}
\def\Res{\operatorname{\br^\lambda_A}}
\def\und{\underline}
\def\Lietg{$A_k(\lie{g})(\bsigma,r)$}

 \nc{\Cal}{\cal} \nc{\Xp}[1]{X^+(#1)} \nc{\Xm}[1]{X^-(#1)}
\nc{\on}{\operatorname} \nc{\Z}{{\bold Z}} \nc{\J}{{\cal J}} \nc{\C}{{\bold C}} \nc{\Q}{{\bold Q}}
\renewcommand{\P}{{\cal P}}
\nc{\N}{{\Bbb N}} \nc\boa{\bold a} \nc\bob{\bold b} \nc\boc{\bold c} \nc\bod{\bold d} \nc\boe{\bold e} \nc\bof{\bold f} \nc\bog{\bold g}
\nc\boh{\bold h} \nc\boi{\bold i} \nc\boj{\bold j} \nc\bok{\bold k} \nc\bol{\bold l} \nc\bom{\bold m} \nc\bon{\bold n} \nc\boo{\bold o}
\nc\bop{\bold p} \nc\boq{\bold q} \nc\bor{\bold r} \nc\bos{\bold s} \nc\boT{\bold t} \nc\boF{\bold F} \nc\bou{\bold u} \nc\bov{\bold v}
\nc\bow{\bold w} \nc\boz{\bold z} \nc\boy{\bold y} \nc\ba{\bold A} \nc\bb{\bold B} \nc\bc{\bold C} \nc\bd{\bold D} \nc\be{\bold E} \nc\bg{\bold
G} \nc\bh{\bold H} \nc\bi{\bold I} \nc\bj{\bold J} \nc\bk{\bold K} \nc\bl{\bold L} \nc\bm{\bold M} \nc\bn{\bold N} \nc\bo{\bold O} \nc\bp{\bold
P} \nc\bq{\bold Q} \nc\br{\bold R} \nc\bs{\bold S} \nc\bt{\bold T} \nc\bu{\bold U} \nc\bv{\bold V} \nc\bw{\bold W} \nc\bz{\bold Z} \nc\bx{\bold
x} \nc\KR{\bold{KR}} \nc\rk{\bold{rk}} \nc\het{\text{ht }}

\nc\toa{\tilde a} \nc\tob{\tilde b} \nc\toc{\tilde c} \nc\tod{\tilde d} \nc\toe{\tilde e} \nc\tof{\tilde f} \nc\tog{\tilde g} \nc\toh{\tilde h}
\nc\toi{\tilde i} \nc\toj{\tilde j} \nc\tok{\tilde k} \nc\tol{\tilde l} \nc\tom{\tilde m} \nc\ton{\tilde n} \nc\too{\tilde o} \nc\toq{\tilde q}
\nc\tor{\tilde r} \nc\tos{\tilde s} \nc\toT{\tilde t} \nc\tou{\tilde u} \nc\tov{\tilde v} \nc\tow{\tilde w} \nc\toz{\tilde z} \nc\woi{w_{\omega_i}}
\nc\chara{\operatorname{Char}}

\title{Fusion Product Structure of Demazure modules}
\author[R. Venkatesh]{R. Venkatesh}
\address{\noindent Department of Mathematics, The Institute of Mathematical Sciences, Chennai 600113, India.}
\email{rvenkat@imsc.res.in}
\begin{abstract} 
Let $\lie g$ be a finite--dimensional complex simple Lie algebra. Given a non--negative integer $\ell$, we define $\mathcal{P^+_\ell}$
to be the set of dominant weights $\lambda$ of $\lie g$ such that $\ell\Lambda_0+\lambda$ is a dominant weight for the corresponding untwisted affine Kac--Moody algebra $\widehat{\lie g}$. 
For the current algebra $\lie g[t]$ associated to $\lie g$, we show that the fusion 
product of an irreducible $\lie g$--module $V(\lambda)$ such that $\lambda\in \mathcal{P^+_\ell}$ and a finite number of special family of $\lie g$--stable Demazure modules of level $\ell$ (considered in \cite{FoL2} and \cite{FoL}) again turns out to be a Demazure module. 
This fact is closely related with several important conjectures. 
We use this result to construct the $\lie g[t]$--module structure of the irreducible $\Lhg$--module $V(\ell \Lambda_0+\lambda)$ as a semi--infinite fusion product of finite dimensional $\lie g[t]$--modules as conjectured in \cite{FoL}.
As a second application we give further evidence to the conjecture on the generalization of Schur positivity (see \cite{chfsa}). 
\end{abstract}
\maketitle

\section*{Introduction} 
The current algebra associated to a simple Lie algebra $\lie g$  is the special maximal parabolic subalgebra of the untwisted affine Lie algebra $\widehat{\lie g}$ associated to $\lie g$. 
As a vector space it is isomorphic to  $\lie g\otimes \mathbb{C}[t]$, where $\mathbb{C}[t]$ is the polynomial ring in the indeterminate $t$ and the Lie bracket is given in the obvious way.
Equivalently, it is just the Lie algebra of polynomial maps  $\mathbb{C}\to \lie g$.
The study of the category of finite dimensional graded representations of the current algebra has been the subject of many articles in the recent years. See for example \cite{eddy}, \cite{BCM}, \cite{chbdz}, \cite{deFK}, \cite{FoL2}, \cite{FoL}, \cite{Naoi}, \cite{Naoi1}. 
One of the original motivations for the study of this category was that it is closely related
to the representation theory of the corresponding quantum affine algebra. The results of \cite{Ckir}
and \cite{CMkir} showed that many interesting families of irreducible representations of the quantum
affine algebra associated to $\lie g$, when specialized to $q = 1$ give indecomposable representations
of the current algebra. For example, these families include the Kirillov--Reshetikhin modules.

There is an important class of finite dimensional graded $\lie g[t]$--modules, called Demazure modules. Let $\lie b$ be a Borel subalgebra of $\lie g$ and $\widehat{\lie b}$ be the corresponding Borel subalgebra of $\Lhg$.
By definition, a Demazure Module is a $\widehat{\lie b}$--submodule of an irreducible highest weight representation of $\widehat{\lie g}$ generated by an extremal weight vector.
We shall only be concerned with $\lie g$--stable Demazure modules in this article. 
These modules naturally become finite dimensional graded modules for the associated current algebra $\lie g[t]$ by restriction. The Demazure module of level $\ell$ and weight $\lambda\in P^+$ for $\lie g[t]$ is denoted by $D(\ell, \lambda)$, where $P^+$ denotes the set of dominant weight of $\lie g$.
The defining relations of these $\lie g[t]$--modules are greatly simplified in \cite{chve} using the results of \cite{mathieu} (see also \cite{FoL}, and \cite{Naoi}). 
We use this simplified presentation to study fusion product structure of $D(\ell, \lambda)$ for some special weights $\lambda\in P^+$.
We assume that $\lie g$ is simply laced for simplicity and state our main result here. We refer the reader to Theorem \ref{mapsdem} for a more general statement.
The primary goal of this paper is the following:
\begin{thm*} Let $\ell, m\in \mathbb{N}$ and $z, z_1, \dots, z_m\in \mathbb{C}$ be given distinct complex numbers. Let $\mu\in P^+$, $\lambda\in \mathcal{P^+_\ell}$ and suppose that there exists $\mu_j\in P^+$, $1\le j\le m$ 
such that $\mu=\mu_1+\cdots+\mu_m$. Then we have an isomorphism of $\lie g[t]$--modules,
$$D(\ell,\ell \mu+\lambda)\cong D^{z_1}(\ell,\ell \mu_1)*\cdots*D^{z_m}(\ell, \ell \mu_m)*\ev_{z}V(\lambda)$$
\end{thm*}
This generalizes a theorem of Fourier and Littelmann (see \cite[Theorem C]{FoL}) where they only consider the case $\lambda=0$. Note in particular that the fusion product of an irreducible module and a finite number of Demazure modules of the form $D(\ell,\ell\mu)$, for sufficiently large $\ell$, is independent of the choice of the parameters that used in the fusion
product construction. Since the Kirillov--Reshetikhin modules are Demazure modules (see \cite{CMkir}, \cite{FoL}), the new interpretations of our
main theorem has many interesting applications related to several important conjectures.
We give two important applications in this article.
Let $V(\ell\Lambda_0+\lambda)$ be the irreducible highest weight module of highest weight $\ell\Lambda_0+\lambda$
for the untwisted affine Kac--Moody $\widehat{\lie g}$, where $\Lambda_0$ is the fundamental weight of $\widehat{\lie g}$ corresponding to the extra node of the extended Dykin diagram of $\lie g$. 
As a first application we construct the $\lie g[t]$--module structure of $V(\ell \Lambda_0+\lambda)$ as a semi--infinite fusion product of finite dimensional $\lie g[t]$--modules as conjectured in \cite{FoL}.
More preciously, we prove the following: \begin{thm*} 
Let $\Lambda = \ell \Lambda_0 + \lambda$ be a dominant integral weight for $\Lhg$, then 
$$
V(\Lambda) \ \text{and} \ \underset{N \rightarrow \infty}{\mbox{lim }} D(\ell,\ell\Theta) \ast \cdots \ast D(\ell,\ell\Theta) \ast V(\lambda)
$$
are isomorphic as $\lie g[t]$--modules.
\end{thm*}

The special case of this theorem (when $\Lambda$ is a multiple of $\Lambda_0$) was proved in [Theorem 9, \cite{FoL}]. The semi--infinite fusion construction was already established 
in \cite{FFsemi} for $\lie g= \lie sl_2$ and in \cite{kedemsemi} for $\lie g=\lie sl_n$.
As a further application we prove some results on the generalization of Schur positivity, that gives some additional evidence to the conjecture that appeared in \cite{chfsa}. We now recall the conjecture on generalization of Schur positivity.
The following  was conjectured in \cite[Section 2.3]{chfsa} by Chari, Fourier and Sagaki:
\begin{conj}\label{main--conj}
Let $\lie g$ be a simple Lie algebra, and $\lambda_1, \lambda_2, \mu_1, \mu_2 \in P^+$ such that 
\begin{enumerit}
 \item[(i)] $\lambda_1+\lambda_2=\mu_1+\mu_2$, 
 \item[(ii)]\label{schurcond} $\min\{\lambda_1(h_\alpha),\lambda_2(h_\alpha)\}\le \min\{\mu_1(h_\alpha),\mu_2(h_\alpha)\}$, for all $\alpha\in R^+$,                                                                                             
\end{enumerit}
then there exists a surjective $\lie g$--module map 
$V(\mu_1) \otimes V(\mu_2)\rightarrow V(\lambda_1) \otimes V(\lambda_2)\rightarrow 0.$
\end{conj}

We give positive answer to this conjecture in certain important cases. 
Indeed we prove stronger statements in this article, we prove that the fusion product of two irreducible $\lie g$--modules in certain cases turns out be Demazure modules and use this fact to get
elegant presentation for $V(\lambda_1)*V(\lambda_2)$ for some special choices of $\lambda_1, \lambda_2$. We use this elegant presentation to prove that there exists surjective $\lie g[t]$--module maps between the appropriate fusion products (see Theorem \ref{genschurpos}).

The paper is organized as follows. Section \ref{preliminaries} has the basic notation and elementary
results needed for the rest of the paper. In section \ref{weyl}, we give the definition of Weyl and Demazure modules and recall their presentations. In section \ref{fusion}, we prove that the fusion product of an irreducible module and the Demazure modules of the same level again turns out to be a Demazure module.
The key steps of the proof are to show that there is a surjective map between appropriate modules and its dimension matches. The existence of the surjective map proved using 
the presentation of the Demazure module that was established in \cite{chve}. 
In section \ref{affine} we calculate the dimension of the Demazure modules using Demazure operators and give the limit construction of $V(\ell\Lambda_0+\lambda)$.
In section \ref{schur} we give further evidence to the conjecture on the generalization of the Schur positivity.

{\em Acknowledgements: The author is very grateful to Vyjayanthi Chari for many useful discussions and encouragement. This work was partially supported by a grant from the Niels Henrik Abel Board.
The author thanks the Niels Henrik Abel Board and the Department of Mathematics at UCR for their financial support. The author acknowledges the hospitality and excellent working
conditions at the Department of Mathematics at UCR where part of this work
was done. The author likes to thank the referee for his useful comments on the previous manuscript that helped to improve the readability of this manuscript.
}

\section{Preliminaries}\label{preliminaries} Our base field will be the complex numbers $\mathbb{C}$ throughout. We let $\mathbb{Z}$ (resp. $\mathbb{Z}_+$, $\mathbb{N}$) be the set of integers (resp. non--negative, positive integers) and $\mathbb{C}[t]$  be the polynomial ring in a variable $t$.

\subsection{} We will broadly follow the notations of \cite{chve}. Let $\lie g$ be an arbitrary finite dimensional simple Lie algebra with Cartan subalgebra $\lie h$ and Borel subalgebra $\lie b.$  
We assume that the rank of $\lie g$ is $n$, i.e., $\dim\lie h=n$. Denote $I=\{1,\dots,n\}$. Let $R$ be the set of roots, $R^+$ the set of positive roots and $\{\alpha_i: i\in I\}$ the set of simple roots of $\lie g$ with respect to $\lie h$.
Let $P,Q,P^+,Q^+$ be the weight lattice, the root lattice, the sets of dominant weights and non-negative integer linear
combinations of simple roots respectively. The Weyl group $W$ of $\lie g$ is generated by the simple reflections $s_i=s_{\alpha_i}$ associated to the simple roots. Let $w_0$ be the longest element of $W$. 
Let $(\ , \ )$ be the non--degenerate, normalized $W$--invariant symmetric bilinear form on $\lie h^*$ such that the square length of a long root is two.
For $\alpha\in R$, $h_\alpha\in \lie h$ denotes the corresponding co--root and set $d_\alpha=2/(\alpha,\alpha)$.
Let $\{\omega_i: i\in I\}\subset\lie h^*$ be the set of fundamental weights, i.e., $(\omega_i,d_j\alpha_j)=\delta_{i,j}$, where $\delta_{i,j}$ is the Kronecker delta symbol.
Let $\lie g_\alpha$ be the root space of $\lie g$ corresponding to the root $\alpha\in R$  and  set $\lie n^\pm=\bigoplus_{\alpha\in R^+}\lie g_{\pm\alpha}$ and $\lie g_{\pm\alpha}=\mathbb{C}x^\pm_{\alpha}$. 
Denote by  $\Theta\in R^+$ be the highest root in $R$ and recall that $[x^\pm_\Theta,\lie n^\pm]=0$. For $i\in I$, we simply  write $x^\pm_i$, $h_i$, $d_i$ for $x^\pm_{\alpha_i}$, $h_{\alpha_i}$, $d_{\alpha_i}$.

\subsection{} For a given Lie algebra $\lie a$, we let $\bu(\lie a)$ be its universal enveloping algebra. The current Lie algebra of $\lie g$ is denoted by $\lie g[t]$. As a vector space it is just $\lie g\otimes \mathbb{C}[t]$ and the Lie bracket is defined in the natural way: $[a\otimes f, b\otimes g]=[a,b]\otimes f g$, for all $a,b\in\lie g$ and $f,g\in\mathbb{C}[t]$.   
Denote by $\lie g[t]_+$ the ideal $\lie g\otimes t\mathbb{C}[t]$ and 
we freely identify $\lie g$ with the Lie subalgebra $\lie g\otimes 1$ of $\lie g[t]$. We clearly have the isomorphism of vector spaces $$\lie g[t]=\lie g[t]_+\bigoplus\lie g, \qquad \bu(\lie g[t])\cong \bu(\lie g[t]_+)\otimes\bu(\lie g).$$ 
The degree grading on $\mathbb{C}[t]$ defines a natural $\mathbb{Z}_+$--grading on $\lie g[t]$ and hence also on  $\bu(\lie g[t])$: an element of the form $(a_1\otimes t^{r_1})\cdots (a_s\otimes t^{r_s})$ has  grade $r_1+\cdots+r_s$. Denote by $\bu(\lie g[t])[r]$ the subspace of grade $r$.

\subsection{} A representation $V$ of $\lie g[t]$ is {\em graded} if it admits a $\mathbb{Z}$--graded vector space decomposition which admits a compatible Lie algebra action of $\lie g[t]$, i.e.,
$$V=\bigoplus_{r\in\mathbb{Z}}V[r],\qquad(\lie g\otimes t^s)V[r]\subset V[r+s],\ \  r\in\mathbb{Z}, \ \ s\in\mathbb{Z}_+.$$  
A morphism between two graded $\lie g[t]$--modules, by definition, is a degree zero morphism of $\lie g[t]$--modules. 
Let $M$ be an $\lie g$--module and $z\in\mathbb{C}$. Define a new $\lie g[t]$--module structure on $M$ with respect to $z$ by: $(a\otimes t^r)m=z^ram$, and denote this module as $\ev_z M$. It is easy to see that $\ev_zM$ is an irreducible $\lie g[t]$--module if and only if $M$ is an irreducible $\lie g$--module. 
Moreover the module $\ev_0M$ is a graded $\lie g[t]$--module and $\left(\ev_0M\right)[0]=M$ and in particular, $\lie g[t]_+(\ev_0M) =0$.

\subsection{} Given $\mu\in P^+$, we let $V(\mu)$ be the irreducible finite--dimensional $\lie g$--module generated by an element $v_\mu$ with defining relations $$x_i^+v_\mu=0,\ \ h_iv_\mu=\mu(h_i)v_\mu,\ \ (x_i^-)^{\mu(h_i)+1}v_\mu=0,\ \ i\in I.$$ It is well--known that any finite--dimensional $\lie g$--module $V$ is isomorphic to a direct sum of irreducible modules $V(\mu)$, $\mu\in P^+$. Further, we have a weight space decomposition with respect to $\lie h$, $$V=\bigoplus_{\nu\in P} V_\nu,\ \ \ \ V_\nu=\{v\in V: hv=\nu(h)v,\ \ h\in\lie h\},$$ and we set $\wt V=\{\nu\in P: V_\nu\ne 0\}$.

\subsection{Affine Lie algebra}
Let $\Lhg$ be the untwisted affine Kac--Moody algebra corresponding
to the extended Dynkin diagram of $\lie g$. $$
\Lhg=\lie g\otimes_\mathbb{C}\mathbb{C}[t,t^{-1}]\oplus \mathbb{C} K\oplus \mathbb{C} d
$$
Here $K$ is the canonical central element and $d$ denotes the derivation $d=t\frac{d}{dt}$. Naturally we consider the Lie algebra $\lie g$
as a subalgebra of  $\Lhg$.
In the same way, $\lie h$ and $\lie b$ are
subalgebras of the Cartan subalgebra $\Lhh$ respectively the Borel
subalgebra $\Lhb$ of $\Lhg$:
\begin{equation*}
\Lhh=\lie h\oplus\mathbb{C} K\oplus\mathbb{C} d,\quad
\Lhb=\lie b\oplus\mathbb{C} K\oplus\mathbb{C} d\oplus\lie g\otimes_\mathbb{C} t\mathbb{C}[t]
\end{equation*}
Let $\delta$ be the positive non--divisible imaginary root of $\Lhg$ with respect to $\Lhb$ and $\ph^+$ be the set of dominant weights of $\Lhg$. Denote $\widehat{W}$ (resp. $\widetilde{W}$) the (resp. extended) affine Weyl group of $\Lhg$.
Then we have $\widetilde{W}=\Sigma\semi\widehat{W}$, where $\Sigma$ is the group of Dynkin diagram automorphisms. We let $\Lambda_0$ be the fundamental weight corresponding to the extra node in the extended Dynkin diagram of $\lie g$.
 We denote $V(\Lambda)$ by  the  irreducible 
$\Lhg$--highest weight module of highest weight $\Lambda \in \ph^+$. 
For an element $\mu \in\bigoplus_{i=1}^n\mathbb{Z}d_i\omega_i$, denote by $t_\mu\in \widetilde{W}$ be the corresponding translation. In particular, we have 
$$t_{\mu}(b\Lambda_0+\lambda)=b\Lambda_0+\lambda+b\mu \ \ \ (\emph{mod} \ \mathbb{C}\delta)$$ for $\lambda\in P^+$ and $b\in \mathbb{C}.$
\subsection{} We conclude this section with the following simple lemma which will be needed later (see \cite[Section 1.3]{chve}), \begin{lem}\label{dalpha} Suppose that   $\lambda\in P^+$ is such that $\lambda=\sum_{i\in I}d_is_i\omega_i$ for some $s_i\in \mathbb{Z}_+$. Then for $\alpha\in R^+$, there exists $s_\alpha\in \mathbb{Z}_+$ such that $\lambda(h_\alpha) =d_\alpha s_\alpha$.\end{lem}

\section{Weyl and Demazure Modules}\label{weyl}
In this section, we recall the definition of local Weyl modules and Demazure modules and state the required results for $\lie g$--stable Demazure modules. 

\subsection{Weyl Modules}\label{localweylxi}  The definition of the local Weyl modules was given originally in \cite{CPweyl} and later in \cite{CFK} and \cite{FL}.
Given $\lambda\in P^+,$ the local Weyl module $W_{\loc}(\lambda)$, is the cyclic $\lie g[t]$--module generated by an element $w_\lambda$ with the following defining relations:  for $i\in I$ and $s\in\mathbb{Z}_+$,  \begin{gather}\label{locweylg}(x^+_i\otimes\mathbb{C}[t])w_\lambda=0,\ \  (h_i\otimes t^s)w_\lambda=\lambda(h_i)\delta_{s,0}w_{\lambda},\\ \label{integloc} (x_i^-\otimes 1)^{\lambda(h_i)+1}w_\lambda=0.\end{gather}  
It is easy to see that $\wt W_{\loc}(\lambda)\subset\lambda-Q^+$ and that $\dim W_{\loc}(\lambda)_\lambda=1$, hence  $W_{\loc}(\lambda)$ is an indecomposable module.
Note {{that $W_{\loc}(0)$ is isomorphic to  the trivial $\lie g[t]$--module}}.
It was proved in \cite{CPweyl} (see also \cite{CFK}) that the local Weyl modules are  finite--dimensional and so, in particular, we have \begin{equation}\label{intega}(x^-_\alpha\otimes 1)^{\lambda(h_\alpha)+1}w_\lambda=0,\ \ \ \alpha\in R^+.\end{equation} 
We declare the grade of $w_\lambda$ to be zero, so that the local Weyl module is graded by $\mathbb{Z}_+$. It follows that $W_{\loc}(\lambda)[0]\cong_{\lie g} V(\lambda),$ and
moreover,  $\ev_0 V(\lambda)$ is the unique graded irreducible quotient of $W_{\loc}(\lambda)$.
\vskip 6pt

\subsection{Demazure Modules} Let us begin with the traditional definition of Demazure modules. For $\Lambda \in \ph^+$ and $w\in\widehat{W}$, the $U(\Lhb)$--submodule $V_w(\Lambda)=U(\Lhb)V(\Lambda)_{w\Lambda}$ is the {\it Demazure submodule} of $V(\Lambda)$ associated to $w$. 
More generally, one can associate a Demazure module with any element of $\widetilde{W}$ as follows: For $\Lambda\in \ph^+$, $\sigma \in \Sigma$ and $w\in \widehat{W}$, define
$$V_{\sigma w}(\Lambda)=V_{\sigma w\sigma^{-1}}(\sigma(\Lambda))\ \ \text{and} \ \ V_{w\sigma}(\Lambda)=V_w(\sigma(\Lambda)).$$
In this article, we are mainly interested in the $\lie g$--stable Demazure modules, which are in particular modules for the associated current algebra $\lie g[t]$.
It is known that $V_w(\Lambda)$ is $\lie g$--stable if and only if $w\Lambda\in -P^++\ell\Lambda_0+\mathbb{Z}\delta$, where $\ell$ is the level of $\Lambda$. For $\lambda\in P^+$, $\ell\in \mathbb{N}$ and $m\in \mathbb{Z}$, there exists a unique $\Lambda\in \ph^+$ such that
$w_0\lambda+\ell\Lambda_0+m\delta\in \widehat{W}\Lambda.$ For an element $w\in \widehat{W}$ such
that $w\Lambda=w_0\lambda+\ell\Lambda_0+m\delta,$ we write $D(\ell, \lambda)[m]=V_w(\Lambda).$ We are indeed interested only in $\lie g[t]$--module structure of $D(\ell, \lambda)[m]$ and it is easy
to see that the $\lie g[t]$--structure of $D(\ell, \lambda)[m]$ is independent of $m$. So it makes sense to denote $D(\ell, \lambda)$ by these isomorphism class of $\lie g[t]$--modules. When $\ell=0$, we simply take $\lambda=0$ and $D(\ell, \lambda)$ is the trivial module in this case.

\vskip 6pt
Now we give one of the equivalent definition of $D(\ell, \lambda)$ (see \cite[Proposition 3.6]{Naoi}) which is appropriate to our study.
We refer the reader to \cite{FoL} and \cite{Naoi} for more details.
The $\lie g[t]$--graded Demazure module $D(\ell,\lambda)$ of level $\ell\in\mathbb{Z}_+$ and weight $\lambda\in P^+$ is the graded quotient of $W_{\loc}(\lambda)$ generated by the elements,
 \begin{gather} \label{demrelo}
 \{(x^-_\alpha\otimes t^p)^{r+1}w_\lambda:\ \ p\in\mathbb{Z}_+,\ \   r\ge\max\{0,\lambda(h_\alpha)-d_\alpha\ell p\},\ \ \text{for} \ \alpha\in R^+\}. \end{gather}
Again, we have $D(\ell,\lambda)[0]\cong_{\lie g} V(\lambda)$. The defining relations of these modules are greatly simplified in \cite{chve}. 
We first fix some notations to state the results of \cite{chve}. Let $\ell\in\mathbb{Z}_+$, $\lambda\in P^+$ be given.
For $\alpha\in R^+$, with $\lambda(h_\alpha)>0$,  let  $s_\alpha,m_\alpha\in\mathbb{N}$ be the unique positive integers so that  $$\lambda(h_{\alpha})=(s_{\alpha}-1)d_{\alpha}\ell+m_{\alpha},\ \ \ 0<m_{\alpha}\le d_{\alpha}\ell,$$ where we recall that $d_\alpha=2/(\alpha,\alpha)$. If $\lambda(h_\alpha)=0$ set $s_\alpha=0=m_\alpha$.
With the notations above, we have the following theorem (see \cite[Theorem 2, Section 3.5]{chve}):

 \begin{thm}\label{genreldem1} Let   $\ell\in\mathbb{Z}_+$ and $\lambda\in P^+$.
  Then $D(\ell,\lambda)$ is the quotient of $W_{\loc}(\lambda)$ by the submodule generated by the elements
 \begin{gather}\label{demrel00}\{ (x^-_{\alpha}\otimes t^{s_{\alpha}})w_\lambda: \alpha\in R^+\}\ \bigcup\  \{ (x^-_{\alpha}\otimes t^{s_{\alpha}-1})^{m_{\alpha}+1}w_\lambda: \alpha\in R^+, \ m_\alpha<d_\alpha\ell\}\ .\end{gather}\end{thm}
\vskip 6pt
 We record the following simple fact for future use.
 \begin{cor}
  Let $\ell\in\mathbb{Z}_+$. For $\lambda\in \mathcal{P^+_\ell}$, $D(\ell,\lambda)\cong \ev_0 V(\lambda)$ as $\lie g[t]$--modules.
 \end{cor}
 \begin{proof}
The condition $\ell \Lambda_0+\lambda\in \ph^+$ implies that $\lambda(h_\Theta)\le \ell$. 
Now we have $\lambda(h_\alpha)\le d_\alpha\lambda(h_\Theta)\le d_\alpha \ell$, for all $\alpha\in R^+$ and hence $s_\alpha\le 1$, for all $\alpha\in R^+$. Now it is immediate from Theorem \ref{genreldem1} that,
 $D(\ell,\lambda)\cong {\rm ev_0} V(\lambda)$ as $\lie g[t]$--modules.
 \end{proof}

\section{Fusion Product Structure of Demazure modules}\label{fusion} Let us begin by recalling   the definition of fusion products of $\lie g[t]$--modules given in \cite{FL}.  
The main result of this section is Theorem \ref{mapsdem} on the fusion product of an irreducible module and the Demazure modules of the same level.
 \subsection{} Let $V$ be a cyclic $\lie g[t]$--module generated by an element $v\in V $.
 We define a filtration  $\{F^rV\}_{r\in\mathbb{Z}_+}$ on $V$ as follows: $$F^rV= \sum_{0\le s\le r}\bu(\lie g[t])[s]v.$$
 Clearly each $F^rV$ is a $\lie g$--module. If we set $F^{-1}V=0$, then the associated graded space $\gr V=\bigoplus_{r=0}^{\infty} F^rV/F^{r-1}V$ acquires a natural structure of a cyclic graded $\lie g[t]$--module with the following action: $$(x\otimes t^s)(\overline w)=\overline{(x\otimes t^s)w},\ \ \ \overline w\in F^rV/F^{r-1}V,$$ 
 where $\overline w$ is the image of $w$ in $\gr V$.
 Moreover, $\gr V\cong V$ as $\lie g$--modules and $\gr V$ is the  cyclic $\lie g[t]$--module generated by $\bar v$. 
 
 The following lemma is simple and very useful.
 \begin{lem}\label{elemfusion} Suppose that $V$ is a cyclic $\lie g[t]$--module generated by $v\in V$. Then for all $u\in V$, $x\in\lie g$, $r\in\mathbb{N}$, $a_1,\dots, a_r\in\mathbb{C}$, we have $$(x\otimes t^r)\overline u= \overline{(x\otimes (t-a_1)\cdots (t-a_r)) u}.$$\hfill\qedsymbol
 \end{lem}

\subsection{} Let $V$ be a $\lie g[t]$--module and $z\in\mathbb{C}$. We define a new $\lie g[t]$--module action on $V$ as follows: $$(x\otimes t^r)v=(x\otimes (t+z)^r)v,\ \ x\in \lie g,\ \  r\in\mathbb{Z}_+,\ \ v\in V.$$ We denote this new module by $V^z$.
 Let  $V_1,\dots, V_m$ be  finite--dimensional graded $\lie g[t]$--modules generated by elements $v_j$, $1\le j\le m$ and  let  $z_1,\dots,z_m$ be   distinct complex numbers. Let  $$\bv= V_1^{z_1}\otimes\cdots\otimes V_m^{z_m},$$ be the corresponding tensor product of $\lie g[t]$--modules.
 It is easily checked (see \cite[Proposition 1.4]{FL}) that the module $\bv$ is cyclic $\lie g[t]$--module and generated by the element 
 $v_1\otimes \cdots\otimes v_m.$ 
 
 \subsection{} The corresponding associated graded $\lie g[t]$--module $\gr \bv$ is defined to be the fusion product of $V_1,\dots ,V_m$ with respect to the parameters $z_1,\dots, z_m$ and denoted by $V_1^{z_1}*\cdots *V_m^{z_m}$.
 Clearly $V_1^{z_1}*\cdots *V_m^{z_m}$ is generated by the image of $v_1\otimes\cdots\otimes v_m$.
In what follows next, for ease of notation, we suppress the dependence of the fusion product on the parameters and just write $V_1*\cdots*V_m$ instead $V_1^{z_1}*\cdots *V_m^{z_m}$. But unless explicitly stated, it should be assumed that the fusion product does depend on these parameters.
Given elements $w_s\in V_s$, $1\le s\le m$, we shall denote by $w_1*\cdots *w_m\in V_1*\cdots*V_m $ the image of the element $w_1\otimes\cdots \otimes w_m\in V_1^{z_1}\otimes\cdots\otimes V_m^{z_m}$.
Here, we record the following simple lemma (see \cite[Section 4.3]{chve}) for future use. \begin{lem}\label{fusionweyl} Let $\lambda_s\in P^+$ and  $V_s$ be a $\lie g[t]$--module quotient of $W_{\loc}(\lambda_s)$  for $1\le s\le m$.   Then $V_1*\cdots*V_m$ is a graded $\lie g[t]$--module quotient of $W_{\loc}(\lambda)$, where $\lambda=\sum_{s=1}^m\lambda_s$.\end{lem}

\subsection{}   
We recall the definition of $\mathcal{P^+_\ell}=\{\lambda\in P^+: \ell \Lambda_0+\lambda\in \ph^+\}$.
Now it is convenient to define a subset $\Gamma$ of $P^+$ as follows: $$\Gamma=\left\{\lambda\in P^+:\ \   \lambda=\sum_{i\in I}d_is_i\omega_i\right\}$$
The following is the statement of our main theorem:
\begin{thm}\label{mapsdem} Let $\mu\in\Gamma$, $\lambda\in P^+$, $\ell\in \mathbb{N}$ and  suppose that there exists $\mu_j\in\Gamma$, $p_j\in \mathbb{N}$, $1\le j\le m$ 
such that $$\ell\mu=p_1\mu_1+\cdots+p_m\mu_m, \qquad \mu(h_\alpha)\ge\sum_{j=1}^m\mu_j(h_\alpha),\ \ \alpha\in R^+.$$ 
There exists a non--zero surjective map of graded $\lie g[t]$--modules, $$D(\ell, \ell\mu+\lambda)\to D(p_1, p_1\mu_1)*\cdots*D(p_m,p_m\mu_m)*D(\ell,\lambda)\to 0,$$
Further more if we assume that $\lambda\in \mathcal{P^+_\ell}$ 
and $p_1=\cdots=p_m=\ell$ then we have an isomorphism 
$$D(\ell,\ell \mu+\lambda)\cong D(\ell, \ell \mu_1)*\cdots*D(\ell, \ell \mu_m)*D(\ell,\lambda)\cong D(\ell,\ell \mu_1)*\cdots*D(\ell, \ell \mu_m)*V(\lambda)$$
and, in particular, we have an isomorphism of $\lie g[t]$--modules
$D(\ell, N\ell\Theta+\lambda)\cong D(\ell, \ell\Theta)^{\ast N}\ast V(\lambda)$.
 \end{thm}

\begin{cor} The fusion product of a finite number of modules of the form $D(\ell,\ell \mu)$, $\mu\in\Gamma$ for a fixed $\ell$, and $V(\lambda)$ is independent of the choice of parameters if $\lambda\in \mathcal{P^+_\ell}$. \end{cor}

\begin{rem} The special case when $\lambda=0$ of the theorem (second statement) and the corollary was proved earlier in \cite[Section 3.5]{FoL} using results of \cite{FoL2}.
We note here, that these papers work with a special family of Demazure modules $D(\ell,\lambda)$. In  the notation of the current paper, they only work with the modules of the form  $D(\ell,\ell\mu)$, $\mu\in \Gamma$.
We remark here that the condition $\ell\Lambda_0+\lambda\in \ph^+$ is equivalent to $\lambda(h_\Theta)\le \ell$.
Thus we have infinite family of such examples.
\end{rem}

 \vskip 6pt
\subsection{} We need some results on the $\lie g$--structure of Demazure modules to prove the second statement of our main theorem \ref{mapsdem}.
We state the required proposition here and defer its proof to section \ref{affine}. Its proof uses ideas that evolved in \cite[Section 3]{FoL2} to prove similar results.

\begin{prop}\label{demprop} Let $\ell\in \mathbb{N}$, $\mu\in\Gamma$ and $\lambda\in \mathcal{P^+_\ell}$. Suppose  that there exists $\mu_j\in\Gamma$, $1\le j\le m$, 
such that $\mu=\mu_1+\cdots+\mu_m, $ then we have
\begin{enumerit}
 \item[(i)] $D(\ell,\ell\mu+\lambda)\cong D(\ell, \ell \mu_1)\otimes \cdots \otimes D(\ell, \ell \mu_m)\otimes D(\ell,\lambda)$ as $\lie g$--modules,
 
 \item[(ii)] in particular, $\dim D(\ell,\ell\mu+\lambda)= \dim D(\ell, \ell \mu_1)\cdots \dim D(\ell, \ell \mu_m) \dim D(\ell,\lambda).$\hfill\qedsymbol

\vskip 6pt

\end{enumerit}

\end{prop}

\noindent{\em{Proof of Theorem \ref{mapsdem}.}} \ Let $v_s\in D(p_s,p_s\mu_s)$ be the image of the generator  $w_{\mu_s}$ of $W_{\loc}(p_s\mu_s)$ for $1\le s\le m$, and $v_{m+1}\in D(\ell,\lambda)$ the image of the generator
$w_{\lambda}$ of $W_{\loc}(\lambda)$. 
Using Lemma \ref{fusionweyl}, we see that there exists a surjective map  of graded $\lie g[t]$--modules,
 $$W_{\loc}(\ell \mu+\lambda)\to D(p_1,p_1\mu_1)*\cdots*D(p_m,p_m\mu_m)*D(\ell,\lambda)\to 0.$$
For each $\alpha\in R^+$, write $\lambda(h_\alpha)=(r_\alpha-1)d_\alpha \ell+m_\alpha$ with $m_\alpha\le d_\alpha \ell$ and $\mu(h_\alpha)=d_\alpha s_\alpha$ (use Lemma \ref{dalpha}).

Now using  Theorem \ref{genreldem1}, it suffices to show that,  
 \begin{gather}\label{demrel} (x_\alpha^-\otimes t^{s_\alpha+r_\alpha})(v_1*\cdots *v_m*v_{m+1})=0,\ \alpha\in R^+,\\
 (x^-_{\alpha}\otimes t^{s_{\alpha}+r_\alpha-1})^{m_{\alpha}+1}(v_1*\cdots *v_m*v_{m+1})=0,\ \alpha\in R^+, \ m_\alpha<d_\alpha\ell.\end{gather}

Write $\mu_j(\alpha)=d_\alpha s_{\alpha,j},\ \ 1\le j\le m,$ (use Lemma \ref{dalpha}) and note that we are given that $$s_\alpha\ge\sum_{j=1}^m s_{\alpha,j},\ \ \alpha\in R^+.$$
Setting    $b_\alpha=s_\alpha-\sum_j s_{\alpha,j}$ and taking $z_1,\dots,z_m, z_{m+1}$ be the parameters involved in the definition of the fusion product, we see that

\begin{gather*}(x^-_\alpha\otimes t^{b_\alpha}(t-z_1)^{s_{\alpha,1}}\cdots(t-z_m)^{s_{\alpha,m}}(t-z_{m+1})^{r_\alpha})(v_1\otimes\cdots \otimes v_m\otimes v_{m+1}) \\
={\sum_{r=1}^{m+1}v_1\otimes\cdots\otimes(x^-_\alpha\otimes (t+z_r)^{b_\alpha}(t-z_1+z_r)^{s_{\alpha,1}}\cdots t^{s_{\alpha,r}}\cdots (t+z_r-z_m)^{s_{\alpha,m}}(t+z_r-z_{m+1})^{r_\alpha}v_r)\otimes\cdots\otimes v_{m+1}}.\end{gather*} 
 Using Theorem \ref{genreldem1},  we see  that the relations $(x_\alpha^-\otimes t^b)v_r=0,\ \ b\ge s_{\alpha,r},$ hold in $D(p_r,\mu_r)$ for $1\le r\le m$, and
$(x_\alpha^-\otimes t^b)v_{m+1}=0, \ \ b\ge r_\alpha$, hold in $D(\ell,\lambda)$. 

\vskip 6pt
\noindent
Hence we have $(x^-_\alpha\otimes t^{b_\alpha}(t-z_1)^{s_{\alpha,1}}\cdots(t-z_m)^{s_{\alpha,m}}(t-z_{m+1})^{r_\alpha})(v_1\otimes \cdots \otimes v_m\otimes v_{m+1})=0. $
Now it follows that $$(x^-_\alpha\otimes t^{s_\alpha+r_\alpha})(v_1\otimes \cdots \otimes v_m\otimes v_{m+1})\in F^{s_\alpha+r_\alpha-1}(\bold V),$$
where $\bold V=D^{z_1}(\ell, \ell \mu_1)\otimes \cdots \otimes D^{z_m}(\ell, \ell \mu_m)\otimes D^{z_{m+1}}(\ell,\lambda).$ This implies
$$(x_\alpha^-\otimes t^{s_\alpha+r_\alpha})(v_1*\cdots*v_m*v_{m+1})=0$$
We now prove the second relation. The earlier calculation shows that,
\begin{gather*}
(x^-_\alpha\otimes t^{b_\alpha}(t-z_1)^{s_{\alpha,1}}\cdots(t-z_m)^{s_{\alpha,m}}(t-z_{m+1})^{r_\alpha-1})^{m_\alpha+1}(v_1\otimes \cdots \otimes v_m\otimes v_{m+1})=\\
v_1\otimes\cdots \otimes v_m\otimes (x^-_\alpha\otimes (t+z_{m+1})^{b_\alpha}(t-z_1+z_{m+1})^{s_{\alpha,1}}\cdots (t+z_{m+1}-z_m)^{s_{\alpha,m}}t^{r_\alpha-1})^{m_\alpha+1} v_{m+1}.
\end{gather*} 
But using Theorem \ref{genreldem1} again, we have the following relations in $D(\ell, \lambda)$ $$(x^-_\alpha\otimes t^{b+r_\alpha-1})^{m_\alpha+1}v_{m+1}=0, \ \text{for} \ b\ge 0.$$ 
which yields $(x^-_\alpha\otimes t^{b_\alpha}(t-z_1)^{s_{\alpha,1}}\cdots(t-z_m)^{s_{\alpha,m}}(t-z_{m+1})^{r_\alpha-1})^{m_\alpha+1}(v_1\otimes \cdots \otimes v_m\otimes v_{m+1})=0.$
Again by similar argument, we have
$$(x_\alpha^-\otimes t^{s_\alpha+r_\alpha-1})^{m_\alpha+1}(v_1*\cdots*v_m*v_{m+1})=0.$$ 
This proves the existence of the surjective map $$D(\ell, \ell \mu+\lambda)\to D(p_1, p_1\mu_1)*\cdots*D(p_m,p_m\mu_m)*D(\ell,\lambda)\to 0.$$
Now the second and third statements of the theorem are immediate from the Proposition \ref{demprop}, since the dimension of the corresponding modules matches.
\hfill\qedsymbol


\section{Proof of Proposition \ref{demprop} and Limit constructions }\label{affine}
The first part of this section occupies the proof of Proposition \ref{demprop}. The proof that we present here is a slight generalization of Theorem 1 and Theorem 1 A in [\cite{FoL2}, Section 3.1]. We also remark that special cases of our proposition was already established in \cite{FoL2}. 
In the second part of this section, we reconstruct the $\lie g[t]$--module structure of the irreducible highest weight $\Lhg$--module $V(\ell\Lambda_0+\lambda)$ as a direct limit of fusion products of Demazure modules. This limit construction was conjectured by Fourier and Littelmann in \cite{FoL} and they proved the conjecture for the special case when $\lambda=0$.
\subsection{}
We begin by recalling the Demazure character formula from \cite[Chapter VIII]{skumar}. Denote $D_w$ the Demazure operator associated with an arbitrary element $w\in \widetilde{W}$. 
For $\Lambda\in \ph^+$, $\sigma \in \Sigma$ and $w\in \widehat{W}$ we have
$$\charc_{\Lhh}\ V_{w\sigma}(\Lambda)=\charc_{\Lhh}\ V_w(\sigma(\Lambda))=D_{w}(e(\sigma\Lambda))=D_{w\sigma}(e(\Lambda))$$ 
We note here that we are only interested in $\lie g$--module structure of Demazure modules, and so in particular that we are interested only in $\lie h$--characters and hence it is enough to calculate the Demazure characters modulo $\delta$. 
We use a few results of \cite{FoL2} to calculate Demazure characters (modulo $\delta$) in what follows next, that allows us to conclude our proposition. We refer the readers to \cite{FoL2} for more details.
The following is simple:
\begin{lem}\label{length}
For $\mu\in \Gamma$, we have
$\ell(t_{w_0\mu}w_0)=\ell(t_{w_0\mu})+\ell(w_0)$, where $\ell(-)$ denotes the extended length function of $\widetilde{W}$. Hence, $D_{t_{w_0\mu}w_0}=D_{t_{w_0\mu}}D_{w_0}$.
\end{lem}

\subsection{Proof of Proposition \ref{demprop}}
We have, by definition, $V_{t_{w_0\mu} w_0}(\ell \Lambda_0+\lambda)=D(\ell, \ell \mu+\lambda)$,
since $$t_{w_0\mu} w_0(\ell \Lambda_0+\lambda)=\ell \Lambda_0+w_0(\ell \mu+\lambda) \ \ \emph{mod} \ \mathbb{Z}\delta.$$
Since $V_{t_{w_0\nu}}(\ell \Lambda_0)=D(\ell, \ell \nu)$, by definition, for any $\nu\in \Gamma$, it is sufficient to prove the following:
$$D_{t_{w_0\mu}w_0}(e(\ell \Lambda_0+\lambda))=e({\ell \Lambda_0})\emph{char}_{\lie h}\overline{V_{t_{w_0\mu_1}}(\ell \Lambda_0)}\cdots \emph{char}_{\lie h}\overline{V_{t_{w_0\mu_m}}(\ell \Lambda_0)}\emph{char}_{\lie h} V(\lambda)$$
where we write $\overline{V_{t_{w_0\mu}}(\ell \Lambda_0)}$ for the Demazure module viewed as $\lie g$--module. 
Now using the Demazure operators and the Lemma \ref{length}, we get
\begin{gather*}
D_{t_{w_0\mu}w_0}(e(\ell \Lambda_0+\lambda))=D_{t_{w_0\mu}}D_{w_0}(e({\ell \Lambda_0+\lambda}))=D_{t_{w_0\mu}}(e({\ell \Lambda_0})\emph{char}_{\lie h}V(\lambda)).
\end{gather*}
Using the Lemma 7 in \cite[Section 3.1]{FoL2} we get
$D_{t_{w_0\mu}}(e({\ell \Lambda_0})\emph{char}_{\lie h}V(\lambda))=D_{t_{w_0\mu}}(e({\ell \Lambda_0}))\emph{char}_{\lie h}V(\lambda)$ and using the Theorem 1' in \cite[Section 3.1]{FoL2} we get
$$D_{t_{w_0\mu}}(e(\ell \Lambda_0))=e({\ell \Lambda_0})\emph{char}_{\lie h}\overline{V_{t_{w_0\mu_1}}(\ell \Lambda_0)}\cdots \emph{char}_{\lie h}\overline{V_{t_{w_0\mu_m}}(\ell \Lambda_0)}.$$
Putting all these together we get our desired result
$$D_{t_{w_0\mu}w_0}(e(\ell \Lambda_0+\lambda))=e({\ell \Lambda_0})\emph{char}_{\lie h}\overline{V_{t_{w_0\mu_1}}(\ell \Lambda_0)}\cdots \emph{char}_{\lie h}\overline{V_{t_{w_0\mu_m}}(\ell \Lambda_0)}\emph{char}_{\lie h}V(\lambda),$$
and hence $$D(\ell, \ell\mu+\lambda)\cong D(\ell, \ell \mu_1)\otimes \cdots \otimes D(\ell, \ell \mu_m)\otimes D(\ell,\lambda),$$ as $\lie{g}$--modules.  
This completes the proof of the proposition. 
\subsection{Limit constructions}\label{seclimitconstruction}

Fix a non--zero dominant weight $\lambda$ of $\lie g$ and $\ell \in \mathbb{N}$ such that $\lambda(h_\Theta)\le \ell$. In this subsection,
we give a proof for the limit construction of the irreducible highest weight $U(\Lhg)$--module $V(\ell \Lambda_0+\lambda)$
as conjectured in \cite{FoL}. Note that, in \cite{FoL2} same authors gave such a construction for $\overline{V(\ell\Lambda_0+\lambda)}$ as a
semi-infinite tensor product of finite dimensional $\lie g$--module.  
We first recall the statement of the semi--infinite fusion product construction:
\begin{thm}\label{semiinfconjecture} 
Let $D(\ell, N\ell\Theta+\lambda)\subset V(\Lambda)$ be the Demazure module of level $\ell$ corresponding to the element $t_{-N\Theta}w_0$ of $\widetilde{W}$. 
Let $w\neq0$ be a $\lie g[t]$--invariant vector of $D(\ell, \ell\Theta).$
Let $\bold{V}^\infty_{\ell, \lambda}$ be the direct limit of $$V(\lambda)\hookrightarrow D(\ell, \ell\Theta)*V(\lambda)\hookrightarrow D(\ell, \ell\Theta)*D(\ell, \ell\Theta)*V(\lambda)\hookrightarrow D(\ell, \ell\Theta)*D(\ell, \ell\Theta)*D(\ell, \ell\Theta)*V(\lambda)\hookrightarrow \cdots $$
where the inclusions are given by $v\mapsto w\otimes v.$ Then
$V(\Lambda) \ \text{and} \ \bold{V}^\infty_{\ell, \lambda}$
are isomorphic as $\lie g[t]$--modules.

\end{thm} 
\begin{proof}
Here we follow the ideas of \cite{FoL}.
By Theorem \ref{mapsdem} we have,
for $z_1\not=z_2 \in \mathbb{C}$ and $N\in \mathbb{Z}_+$, an isomorphism of $\lie g[t]$--modules
$$ D(\ell, (N+1)\ell\Theta+\lambda) \cong D(\ell, \ell \Theta)^{z_2} \ast D(\ell, N\ell\Theta+\lambda)^{z_1}.$$ Using this isomorphism of Demazure modules, the assertion can be proved in exactly the same way as 
\cite[Theorem 9]{FoL}. We refer the readers to \cite{FoL} for more details.
\end{proof}

\section{Application to Schur positivity}\label{schur}

In this section, we recall the definition of Kirillov--Reshetikhin modules and its connection with Demazure modules. It has been proved in  \cite{CMkir} (see also \cite{FoL} and \cite{chve}) that Kirillov--Reshetikhin modules are indeed Demazure modules. We use this fact to establish the possible connections of our main theorem with the Schur positivity conjecture. We will freely use the notation established in the previous sections.  

\subsection{} We now recall the definition of the Kirillov--Reshetikhin modules from \cite[Section 2]{CMkir}. Thus given $i\in I$ and $m\in\bz_+$, the Kirillov--Reshetikhin module $KR(m\omega_i)$ is the quotient of the  $W_{\loc}(m\omega_i)$ by the submodule generated by the element  $(x_i^-\otimes t)w_{m\omega_i}$. 
We shall need the following isomorphism of $\lie g[t]$--modules (see \cite[Proposition 5.1]{chve}). 
\begin{prop}\label{dem=kr}
For $i\in I$ and $\ell\in \mathbb{Z}_+$, we have $D(\ell, d_i\ell\omega_i)\cong KR(d_i\ell\omega_i)$ as $\lie g[t]$--modules. Further more if we assume that 
$d_i\omega_i(h_\alpha)\le d_\alpha, \ \alpha\in R^+$, then we have $D(\ell, d_i\ell\omega_i)\cong KR(d_i\ell\omega_i)\cong \ev_0V(d_i\ell\omega_i)$ as $\lie g[t]$--modules. 
\end{prop}

\begin{rem}\label{minusculecoweight}
We remark here that the isomorphism between the Kirillov--Reshetikhin modules and the Demazure modules
was proved earlier in \cite[Section 5]{CMkir} and \cite[Section 3.2]{FoL}. The condition $d_i\omega_i(h_\alpha)\le d_\alpha, \ \alpha\in R^+$ is equivalent to saying that $d_i\omega_i(h_\Theta)\le 1.$ This forces $d_i=1$ for all such $d_i\omega_i.$ Indeed we have the following list of all such $d_i\omega_i$'s:
here, we follow the numbering
of vertices of the Dynkin diagram for $\lie{g}$ in \cite{Bour}.
\begin{equation*}
\begin{array}{ll}
A_{n} & \omega_{i},\, 1 \le i \le n \\[1mm]
B_{n} & \omega_{1} \\[1mm]
C_{n} & \omega_{n} \\[1mm]
D_{n} & \omega_{1},\,\omega_{n-1},\,\omega_{n} \\[1mm]
E_{6} & \omega_{1},\,\omega_{6} \\[1mm]
E_{7} & \omega_{7}
\end{array}
\end{equation*}
\end{rem}

\subsection{}New interpretations of Theorem \ref{mapsdem} using the isomorphism between Kirillov--Reshetikhin modules and Demazure modules has many interesting consequences. We fix a non--negative integer $\ell$ in what follows next in this section.
The following is an immediate consequence of Theorem \ref{mapsdem} and Proposition \ref{dem=kr}.
\begin{thm}\label{krdecom}
 Let $\lambda\in \mathcal{P^+_\ell}$ and $\mu=\sum\limits_{i\in I}d_is_i \omega_i \in \Gamma$. Then we have the following isomorphism of $\lie g[t]$--modules,
 $$D(\ell, \ell\mu+\lambda)\cong KR(d_1\ell\omega_1)^{*s_1}*\cdots *KR(d_n\ell\omega_n)^{*s_n}*V(\lambda).$$
In particular, the fusion product on the right hand side is independent of the choice of the parameters.
\end{thm}
\vskip 6pt

\noindent
We use Theorem \ref{krdecom} to establish an elegant presentation for the module 
$KR(d_i\ell\omega_i)^{\ast k}\ast V(\lambda)$, $\lambda\in \mathcal{P^+_\ell}$.

\begin{prop}\label{demrelfuskr}
 Fix $i\in I$ and let $k\in \mathbb{N}$ and $\lambda\in \mathcal{P^+_\ell}$. Then the module $KR(d_i\ell \omega_i)^{\ast k}\ast V(\lambda)$ is the quotient of $W_{\loc}(kd_i\ell\omega_i+\lambda)$ by the submodule generated by the elements
 \begin{gather}\label{demrelschur}\{ (x^-_{\alpha}\otimes t^{ks_{\alpha}+1})w_{kd_i\ell\omega_i+\lambda}: \alpha\in R^+\}\ \bigcup\  \{ (x^-_{\alpha}\otimes t^{ks_{\alpha}})^{\lambda(h_\alpha)+1}w_{kd_i\ell\omega_i+\lambda}: \alpha\in R^+\}\ ,\end{gather}
 where $s_\alpha:=\frac{1}{d_\alpha}d_i\omega_i(h_\alpha)$ for all $\alpha\in R^+,$ and hence the module $KR(d_i\ell \omega_i)^{\ast k}\ast V(\lambda)$ is independent of the choice of parameters.
 \end{prop}
 
 \begin{proof}
  Using Theorem \ref{krdecom}, we get $D(\ell, kd_i\ell \omega_i+\lambda)\cong KR(d_i\ell\omega_i)^{\ast k}*V(\lambda)$ as $\lie g[t]$--modules. Since $\lambda \in \mathcal{P^+_\ell}$, we have
$\lambda(h_\alpha)\le d_\alpha \ell$ for all $\alpha\in R^+$, now the proposition is immediate from Theorem \ref{genreldem1}.
 \end{proof} 

\subsection{}
When we specialize Theorem \ref{krdecom} to a fundamental weight that is listed in the remark \ref{minusculecoweight}, we get simplified presentation for the fusion products of many interesting family of irreducible representations of $\lie g$. 
We note here that the fusion products of these special family of irreducible $\lie g$--modules are indeed Demazure modules. In particular we get presentation of 2--fold fusion product $V(\lambda)*V(\mu)$ for many special choices of $\lambda, \mu\in P^+$ using this fact, which allows us to provide further evidence to the conjecture on Schur positivity stated in \cite{chfsa}.
For example, we provide presentation for $V(\ell\omega_i)*V(m\omega_j)$ when $\lie g=\lie sl_{n+1}$. 
We remark here that these presentations were previously known only for some special cases, but not in this generality. For example see \cite{FFq-char} for the case $\lie g=\lie sl_2$, and \cite{chve},\cite{FoL} for the case $\ell=m$ and $\lie g=\lie sl_{n+1}$ general.
 \begin{prop}\label{2-fold}
  Let $i\in I$ be such that $d_i\omega_i(h_\Theta)\le 1$ and $\lambda\in \mathcal{P^+_\ell}$. Then the module $V(d_i\ell \omega_i)\ast V(\lambda)$ is the quotient of $W_{\loc}(d_i\ell\omega_i+\lambda)$ by the submodule generated by the elements
 $$\{ (x^-_{\alpha}\otimes t^{2})w_{d_i\ell\omega_i+\lambda}: \alpha\in R^+\}\ \bigcup\  \{ (x^-_{\alpha}\otimes t)^{\emph{min}\{d_i\ell \omega_i(h_\alpha),\ \lambda(h_\alpha)\}+1}w_{d_i\ell\omega_i+\lambda}: \alpha\in R^+\}\ .$$  
 \end{prop}
 \begin{proof}
  Let $z_1, z_2\in \mathbb{C}$ be two distinct complex numbers that used in the definition of the fusion product $V(d_i\ell\omega_i)*V(\lambda)$ and let $v_i*v_\lambda$ be the image of $v_{d_i\ell\omega_i}\otimes v_\lambda$ in $V(d_i\ell\omega_i)*V(\lambda)$.
  Then for any $x\in \lie g$, by the definition of the evaluation representation, we get $(x\otimes (t-z_1)(t-z_2))(v_{d_i\ell\omega_i}\otimes v_{\lambda})=0.$ Thus we get $(x^-_{\alpha}\otimes t^{2})(v_i*v_\lambda)=0$ for any $\alpha\in R^+.$
Now since $(x^-_{\alpha}\otimes t^{2})(v_i*v_\lambda)=0$ implies $(x^-_{\alpha}\otimes t^{3})(v_i*v_\lambda)=0$ for any $\alpha\in R^+,$ the result is immediate from Proposition \ref{demrelfuskr}.  \end{proof}

\vskip 6pt

\noindent

 \begin{cor}
  Let $\lie g$ be a simple Lie algebra with a fundamental weight $\omega_i$ such that $d_i\omega_i(h_\Theta)\le 1$. Let $j \in I$ and $\ell, m\in \mathbb{N}$ such that $\ell\ge m.$ We also assume that
\begin{itemize}
      \item if $\lie g=B_n \ \text{and} \ j\neq 1$ then $\ell\ge 2m,$ 
      \item if $\lie g=C_n \ \text{and} \ j\neq n$ then $\ell\ge 2m,$ 
      \item if $\lie g=D_n \ \text{and} \ j\not\in \{1,n-1,n\}$ then $\ell\ge 2m,$ 
      \item if $\lie g=E_6$ and $j\in \{2,3,5\}$ then $\ell\ge 2m$; for $j=4$, assume that $\ell\ge 3m,$
      \item if $\lie g=E_7$ and $j\in \{1,2,6\}$ then $\ell\ge 2m$; for $j=3,5$, assume that $\ell\ge 3m;$ for $j=4$, assume that $\ell\ge 4m.$
\end{itemize}Then the module $V(d_i\ell \omega_i)\ast V(d_jm\omega_j)$ is the quotient of $W_{\loc}(d_i\ell\omega_i+d_jm\omega_j)$ by the submodule generated by the elements
 $$\{ (x^-_{\alpha}\otimes t^{2})w_{d_i\ell\omega_i+d_jm\omega_j}: \alpha\in R^+\}\ \bigcup\  \{ (x^-_{\alpha}\otimes t)^{\emph{min}\{d_i\ell\omega_i(h_\alpha),\ d_jm\omega_j(h_\alpha)\}+1}w_{d_i\ell\omega_i+d_jm\omega_j}: \alpha\in R^+\}\ .$$
\end{cor}
 \begin{proof}
  It is easy to see that $d_jm\omega_j(h_\Theta)\le \ell$ in all considered cases. Thus the corollary follows from the above Proposition \ref{2-fold} by setting $\lambda=d_jm\omega_j$.
 \end{proof}
\vskip 6pt
\noindent
\subsection{} Proposition \ref{2-fold} and its corollary motivates us to consider the modules $\bold{V}(\lambda, \mu)$, for $\lambda, \mu\in P^+$ that defined as follows: it is the graded quotient of $W_\emph{loc}(\lambda+\mu)$ generated by the elements
 $$\{ (x^-_{\alpha}\otimes t^{2})w_{\lambda+\mu}: \alpha\in R^+\}\ \bigcup\  \{ (x^-_{\alpha}\otimes t)^{\emph{min}\{\lambda(h_\alpha),\mu(h_\alpha)\}+1}w_{\lambda+\mu}: \alpha\in R^+\}\ .$$
 We remark that these modules studied by Fourier in \cite{Fu} for the case $\lie g=\lie sl_{n+1}$. These modules are naturally related to the conjecture on the generalization of Schur positivity of \cite{chfsa}. In fact it is easy to see that this conjecture is true for these modules. The next lemma tells us its connection with fusion products.
 \begin{lem}\label{2-foldlem}
 Let $\lambda, \mu\in P^+$ and $z_1, z_2\in \mathbb{C}$ be two distinct complex numbers. Then there exists a surjective map of $\lie g[t]$--modules
 $$\bold{V}(\lambda, \mu)\rightarrow \ev_{z_1}V(\lambda)*\ev_{z_2}V(\mu)\rightarrow 0$$
 \end{lem}
\begin{proof}
 The lemma easily follows from the definition of the evaluation representation and fusion products. For example, we prove that $(x^-_{\alpha}\otimes t)^{\emph{min}\{\lambda(h_\alpha),\mu(h_\alpha)\}+1}(v_{\lambda}*v_{\mu})=0$ holds in $\ev_{z_1}V(\lambda)*\ev_{z_2}V(\mu)$, where $v_{\lambda}*v_{\mu}$
is the image of $v_{\lambda}\otimes v_\mu$ in $\ev_{z_1}V(\lambda)*\ev_{z_2}V(\mu)$. Indeed, 
$$(x^-_{\alpha}\otimes (t-z_2))^{\lambda(h_\alpha)+1}(v_{\lambda}\otimes v_{\mu})=(z_1-z_2)^{\lambda(h_\alpha)+1}((x^-_{\alpha})^{\lambda(h_\alpha)+1}v_\lambda\otimes v_\mu)=0.$$
Hence $(x^-_{\alpha}\otimes t)^{\lambda(h_\alpha)+1}(v_{\lambda}* v_{\mu})=0$, and also similar argument shows that $(x^-_{\alpha}\otimes t)^{\mu(h_\alpha)+1}(v_{\lambda}* v_{\mu})=0$.
\end{proof}

\begin{prop}\label{schurpostivity}
 Let $i\in I$ be such that $d_i\omega_i(h_\Theta)\le 1$ and $\lambda\in \mathcal{P^+_\ell}$. Suppose that there exists $\mu_1, \mu_2\in P^+$ such that 
 \begin{enumerit}
 \item[(i)] $d_i\ell\omega_i+\lambda=\mu_1+\mu_2$ and 
 \item[(ii)]\label{schurcond} $\min\{\mu_1(h_\alpha),\mu_2(h_\alpha)\}\le \min\{d_i\ell\omega_i(h_\alpha),\lambda(h_\alpha)\}$, for all $\alpha\in R^+$,                                                                                             
\end{enumerit}
then there exists a surjective $\lie g[t]$--module map 
$V(d_i\ell\omega_i)*V(\lambda)\rightarrow V(\mu_1)* V(\mu_2)\rightarrow 0.$
\end{prop}

\begin{proof}
 By Proposition \ref{2-fold}, we get a surjective map of $\lie g[t]$--modules $V(d_i\ell\omega_i)*V(\lambda)\rightarrow \bold{V}_{\mu_1,\mu_2}\rightarrow 0,$
and now using Lemma \ref{2-foldlem} we have a surjective map of $\lie g[t]$--modules $\bold{V}_{\mu_1,\mu_2}\rightarrow V(\mu_1)* V(\mu_2)\rightarrow 0.$ Putting all these together we get a surjective 
$\lie g[t]$--module map 
$$V(d_i\ell\omega_i)*V(\lambda)\rightarrow V(\mu_1)* V(\mu_2)\rightarrow 0.$$
\end{proof}

 \begin{cor}
  Let $\ell, m$ be positive integers such that $m\le \ell \le 2m-1$. Then for any dominant weight $\lambda$ of $\lie{g}$ with $\lambda(h_\Theta)\le 2m-\ell$ and a fundamental weight $\omega_i$ listed in the remark \ref{minusculecoweight}, we have a surjective map of $\lie{g}[t]$--modules
  \begin{gather*}
   V(d_im\omega_i)\ast V(\lambda+(\ell-m)d_i\omega_i)\rightarrow V(d_i\ell\omega_i)\ast V(\lambda) \rightarrow 0.
  \end{gather*}
 \end{cor}
\vskip 6pt

\subsection{} Here is an another important application of Proposition \ref{demrelfuskr}.
\begin{thm}\label{genschurpos}
 Fix $k\in \mathbb{Z}_+$ and $\ell, m\in \mathbb{N}$ such that $\ell\ge m$. Suppose that there exists $\lambda\in P^+$ and $\mu\in \mathcal{P}^+_m$ such that $kd_i\ell\omega_i+\lambda=kd_im\omega_i+\mu$.
\begin{itemize}
 \item[(i)]  Then there exists a surjective map of $\lie{g}[t]-$modules
  \begin{gather*}
   KR(d_im\omega_i)^{\ast k}\ast V(\mu)\rightarrow KR(d_i\ell\omega_i)^{\ast k}\ast V(\lambda) \rightarrow 0.
  \end{gather*}

\item[(ii)] Further more if we assume that $i\in I$ such that $d_i\omega_i(h_\Theta)\le 1$. Then there exists a surjective map of $\lie{g}[t]-$modules
\begin{gather*}
   V(d_im\omega_i)^{\ast k}\ast V(\mu)\rightarrow V(d_i\ell\omega_i)^{\ast k}\ast V(\lambda) \rightarrow 0.
  \end{gather*}
\end{itemize}

\end{thm}

\begin{proof}
 Observe that $kd_i\ell\omega_i(h_\Theta)+\lambda(h_\Theta)=kd_im\omega_i(h_\Theta)+\mu(h_\Theta)$ and $\mu(h_\Theta)\le m$ implies $\lambda\in \mathcal{P^+_\ell}$.
 Now the theorem is immediate from the Proposition \ref{demrelfuskr}.
\end{proof}

\begin{rem}
We end this section with an important remark. We first recall the results of \cite{chfsa}.
In \cite{chfsa}, the authors defined a partial order $\preceq$ on the set of $k$--tuples of dominant integral weights which add up to $\lambda$ (say $P^+(\lambda,k)$): one just requires the inequality $(ii)$ in the Proposition \ref{schurpostivity} to hold for all partial sums.
Given an element of $P^{+}(\lambda,\,k)$, they 
considered the tensor product of the corresponding
simple finite--dimensional $\lie g$--modules and showed that the dimension of this tensor product increases along this partial order $\preceq$.
They also showed that in the case when
$\lambda$ is a multiple of a fundamental minuscule weight
($\lie g$ and $k$ are general) or
if $\lie g$ is of type $A_{2}$ and $k=2$ ($\lambda$ is general),
there exists an inclusion of tensor products along with the partial order $\preceq$ on $P^{+}(\lambda,\,k)$ (see \cite[Theorem 1]{chfsa}). In particular,  if $\lie g$
is of type $A_n$,  this means that the difference of
the characters is Schur positive.

We can recover these results when $\lambda=Nd_i\omega_i$ using Proposition \ref{schurpostivity}, where $\omega_i$ is in the list of remark \ref{minusculecoweight}
and $\lie g$, $k$ are general. We indeed prove that there exists a surjective map between the appropriate tensor products. The existence of such a surjective map between the tensor products can be reduced to the case when $k=2$ and this case easily follows from proposition \ref{schurpostivity}. We refer the readers to \cite{chfsa} for more details.  We note here that the proof in \cite{chfsa} uses the combinatorics of LS paths and our approach avoid this. 
\end{rem} 

\end{document}